\documentclass[12pt,oneside,a4papper]{osajnl}
\usepackage{amsthm}
\usepackage{comment}
\journal{jocn} 

\setboolean{shortarticle}{false}
\title{\centering 
Bilinear control to trajectories of 1D degenerate parabolic equations in moving domains
}

\author[1,2,*]{\centering Alfredo S. Gamboa}
\author[3]{André da Rocha Lopes}
\author[4]{Luis P. Yapu}

\affil[1]{Universidade do Estado do Rio de Janeiro, Escola Politécnica, Nova Friburgo, Brazil}
\affil[2]{Universidad Privada Boliviana, Departamento de Ciencias Exactas, Cochabamba, Bolivia}
\affil[3]{Universidade do Estado do Rio de Janeiro, Instituto de Matemática e Estatística, Rio de Janeiro, Brazil}
\affil[4]{Universidade Federal Fluminense, Instituto de Matemática e Estatística, Niterói, Brazil}

\affil[*]{\centering Contato: alfredo.soliz@iprj.uerj.br}

 \dates{\centering Compiled \today}

\doi{}

\begin{abstract}
In this paper, we are concerned with local controllability properties of degenerate parabolic equations in bounded domains that evolve in time. More precisely, we deal with the exact controllability to a positive trajectory of a one-dimensional semilinear degenerate equation governed via the coefficient of the reaction term. We apply a well-known local inversion method combined with some appropriate specific estimates. 
\end{abstract}

\definecolor{mygreen}{RGB}{44,162,67}
\definecolor{mylilas}{RGB}{186,85,211}
\setboolean{displaycopyright}{false}
\newcommand{\dpar}[2]{\frac{\partial #1}{\partial #2}}

\newcommand{\dpd}[2]{\frac{\partial^2 #1}{\partial #2^2}}

\newcommand{\cara}{\mathbb{1}}

\newtheorem{lema}{Lemma}
\newtheorem{teo}{Theorem}
\newtheorem{myth}{Theorem}
\newtheorem{mydef}{Definition}
\newtheorem{propo}{Proposition}
\newtheorem{coro}{Corollary}
\newcommand{\R}{\mathbb{R}}

\setboolean{displaycopyright}{false}

\begin{document}

\maketitle

\textbf{MSC Classification (2020)}: Primary: 35K65, 93B05; Secondary: 93C10. 

\textbf{keywords}: Degenerate parabolic equations, Moving domains, Controllability, Nonlinear systems in Control Theory, Carleman inequalities.

\section*{Introduction}

\qquad Let $\Omega \subset \mathbb{R}$ be an open bounded set with boundary $\Gamma =\partial \Omega$. For each $T>0$, we denote the cylinder $Q:= \Omega \times (0,T)$ with lateral boundary $\Sigma:= \Gamma \times(0,T)$ and assume that $\omega \subset \Omega$  is a nonempty open set. 

In the present article, we investigate the controllability properties of a moving-boundary problem for a semilinear degenerate parabolic equation governed by a bilinear control in a non-cylindrical space-time domain. To formulate the problem, let us consider $\ell :[0,T] \mapsto (0,\infty)$, a continuously differentiable function, and
$$
    \Omega_t =\{ \overline{x} \in (0,\ell (t)); \, \,\overline{x}= \ell (t)x, \,\, \text{for} \,\, x \in \Omega \},\, 0<t<T.
$$

For $t=0$, we identify $\Omega_0$ with $\Omega$ and assume that each point $\overline{x}$ of the original given domain $\Omega_0$ moves itself through a curve $t\mapsto \Omega_t \subset \mathbb{R}$. We deal
with systems where the control function acts on a set of the form $\widehat{\omega} \times (0,T)$, where $\widehat{\omega} \subset \Omega_t$ is an arbitrarily small open set. More precisely, the following semilinear parabolic system will be considered:
\begin{equation}\label{eq:PDE}
\left\{
\begin{array}
    [c]{lll}%
    u_t - \left(a(\overline{x}) u_{\overline{x}}\right)_{\overline{x}} + F(\overline{x},t,u)=\widehat{h}\cara_{\widehat{\omega}}u & \mbox{in} &
    \widehat{Q}:=\displaystyle\bigcup_{0\leq t \leq T} \{\Omega_t \times \{t\} \},\\
    u(0,t)=u(\ell(t),t)=0 & \mbox{on} & \widehat{\Sigma}:= \displaystyle\bigcup_{0\leq t \leq T} \{\Gamma_t \times \{t\} \},\\
    u(\overline{x},0)=u_{0}(\overline{x}) & \mbox{in} & \Omega,
\end{array}
\right.  %
\end{equation}
where $u = u(\overline{x}, t)$ denotes the associated state, $u_0$ is the initial data, $\widehat{h}=\widehat{h}(\overline{x}, t)$ is the control function, $\cara_{\widehat{\omega}}$ represents the
characteristic function of $\widehat{\omega}$ and $a$
is a diffusion coefficient which degenerates at the extremity $\overline{x}=0$. A model example of such a degenerate coefficient $a$ is the function $a(\overline{x})=\overline{x}^{\alpha}$, for $\alpha \in (0,1)$. Throughout the whole paper, the following hypotheses will be assumed:
\begin{itemize}
    \item[{\textbf {H1}}] $a \in C([0,\ell (t)]) \cap C^1 ((0,\ell(t)])$ satisfying $a(0)=0$, $a>0$ on $(0,\ell(t)]$, $a' \geq 0$,
    %\color{red}
    \begin{equation}\label{eq:cond_a}
        a(f(t)y)=g(t)a(y),
    \end{equation}
    for positive functions $f$ and $g$, 
    %\color{black}
    and $$\overline{x}a'(\overline{x})\leq Ka(\overline{x}),\,\forall \overline{x} \in [0,\ell(t)]\,\,\text{and some}\,\,K \in [0,1).$$
    
    \item[{\textbf {H2}}] F is a $C^1$ function , with bounded derivatives, satisfying $F(\overline{x},t,0)=0$ and 
    \begin{equation*}\label{hf}
    \vert F(\overline{x},t,r_1)-F(\overline{x},t,r_2)-D_3 F(\overline{x},t,r_2)(r_1 -r_2)\vert \leq C \vert r_1 -r_2\vert^2 \,\, \text{for any} \,\, (r_1 ,r_2) \in \mathbb{R} \times \mathbb{R}.
    \end{equation*}
    
    %\color{red}
    \item[{\textbf {H3}}] The positive function $\ell(t)$ satisfies $$\frac{\ell '(t)}{\ell(t)} \leq C,\,\,\text{for some}\,\, C>0,$$ and the function $b(t)$ defined by $b(t)=\frac{g(t)}{\ell(t)^2}$ is positive and satisfies $$\frac{b'(t)}{b(t)} \leq C_b,\,\,\text{for some}\,\, C_b>0.$$
    \end{itemize}
%\color{black}

In order to deal with the controllability properties of problem (\ref{eq:PDE}), it is necessary to introduce the following Sobolev weighted spaces
\begin{equation*}
    \begin{split}
        H^1_a(0,\ell (t)):=&\left\{ u\in L^2(0,\ell (t)) \ : \ u \ \text{is absolutely continuous in} \ \ (0,\ell (t)], \right. \\
        &\left. \quad \sqrt{a}u_{\overline{x}}\in L^2(0,\ell (t)) \ \text{and} \  u(0)=u(\ell (t))=0 \right\}        
    \end{split}
\end{equation*}
and 
$$H^2_a(0,\ell (t)):=\left\{u\in H^1_a(0,\ell (t)) \ : \ au_{\overline{x}}\in H^1(0,\ell (t)) \right\},$$
with respective norms
$$
\Vert u\Vert^{2}_{H^1_a(0,\ell (t))}:= \Vert u\Vert^{2}_{L^2(0,\ell (t))} + \Vert \sqrt{a}u_{\overline{x}}\Vert^{2}_{L^2(0,\ell (t))} \,\, \text{and} \,\,\Vert u\Vert^{2}_{H^2_a(0,\ell (t))}:= \Vert u\Vert^{2}_{H^1_a(0,\ell (t))} + \Vert (au_{\overline{x}})_{\overline{x}}\Vert^{2}_{L^2(0,\ell (t))}\,.
$$

In recent years there has been renewed interest in problems related with partial differential equations
formulated in domains that change in time. This is partly due to the fact that a number of problems in
mathematics are naturally posed in domains with moving boundaries, see for instance \cite{hh-00,lmz-02,gp-99} and references therein. Degenerate parabolic equations in moving boundaries are motivated by the need to model complex physical phenomena where diffusion vanishes at certain points or regions, typically leading to free boundaries or interfaces that evolve over time. These equations are crucial for understanding systems where transport is not uniform throughout the domain.

The degeneracy of the diffusion coefficient reflects the presence of regions where diffusion becomes weak or vanishes completely. This feature is typical in porous media, where permeability may decrease near impermeable boundaries, or in thermal processes involving materials with spatially varying conductivity. In such cases, the degeneracy induces a strong anisotropy in the diffusion process, significantly affecting the propagation of heat, mass, or other quantities.

In addition, the time dependence of the domain $\Omega_t$ models systems with moving boundaries. This framework is relevant in applications such as phase transition problems (e.g., melting and solidification), tumor growth, population dynamics in expanding habitats, and fluid flows in deformable regions. In these scenarios, the evolution of the boundary plays an important role in the overall dynamics of the system and must be explicitly taken into account.

The problem under consideration is more complex than the degenerate equation in a fixed domain since the domain itself is evolving, requiring a different approach to solve the partial differential equation. To solve the controllability problem of \eqref{eq:PDE}, we will construct a diffeomorphism that maps $\widehat{Q}$ onto $Q$. In that way, for each $t \in [0,T]$, we consider a family of functions $\{{\tau}_{t}\}_{0\leq t \leq T}$, where ${\tau}_{t}$ is a deformation of $\Omega$ into an open bounded set ${\Omega}_{t}$ of $\mathbb{R}$. We make the following assumptions on the function ${\tau}_{t}$:
\begin{itemize}
    \item For all $t \in [0,T]$, $\tau_t$ is a $C^2$-diffeomorphism from $\Omega$ to $\Omega_t$,
    \item $\tau_{t}$ has the regularity $C^{1}([0,T];C^{0}(\overline{\Omega},\mathbb{R}))\cap C^{0}([0,T];C^{2}(\overline{\Omega},\mathbb{R}))$.
\end{itemize}

The desired diffeomorphism is given as follows:
$$
    \tau_{t}^{-1}:\widehat{Q} \rightarrow Q \quad \text{defined by} \quad (\overline{x},t) \in \widehat{Q} \rightarrow (x,t) \in Q\,,\quad \text{where} \quad \overline{x}=\tau_t(x):=\ell(t)x\,.
$$

The notion of local controllability that we consider in this work is defined as follows.
\begin{mydef}
    It is said that (\ref{eq:PDE}) is locally null controllable at time $T$ if there exists $\varepsilon >0$ such that, for any $u_0 \in H_{a}^{1}(\Omega)$ with $$\Vert u_0\Vert_{H_{a}^{1}(\Omega)} \leq \varepsilon\,,$$ there exists at least a control function $\widehat{h} \in L^2 (\widehat{\omega} \times (0,T))$ such that the associated state $u$ satisfies
    \begin{equation}\label{nc}
    u(\cdot ,T)=0 \,\,\,\text{in}\,\,\, \Omega_T \,.
    \end{equation}
\end{mydef}

The goal of this work is to study the controllability properties of (\ref{eq:PDE}) for a positive trajectory in the following sense: We consider a positive trajectory $\widetilde{u} \in L^2 (0,T;H^{2}_{a}(\Omega_t))$. That is, there exists a constant $C >0$ such that $\vert \widetilde{u}\vert \geq C>0$, with $\widetilde{u}$ satisfying the uncontrolled equation
\begin{equation}\label{traj.prob1}
    \left\{
    \begin{array}
    [c]{lll}%
    \widetilde{u}_{t}- \left(a(\overline{x}) \widetilde{u}_{\overline{x}}\right)_{\overline{x}} + F(\overline{x},t,\widetilde{u})=0 & \mbox{in} &
    \widehat{Q}:=\displaystyle\bigcup_{0\leq t \leq T} \{\Omega_t \times \{t\} \},\\
    \widetilde{u}(0,t)=\widetilde{u}(\ell(t),t)=0 & \mbox{on} & \widehat{\Sigma}:= \displaystyle\bigcup_{0\leq t \leq T} \{\Gamma_t \times \{t\} \},\\
    \widetilde{u}(\overline{x},0)=\widetilde{u}_{0}(\overline{x}) & \mbox{in} & \Omega.
    \end{array}
    \right.  %
\end{equation}

\begin{mydef}
    It is said that (\ref{eq:PDE}) is locally exactly controllable to the positive trajectory $\widetilde{u}$ at time $T$ if there exists $\varepsilon >0$ such that, for any $u_0 \in H_{a}^{1}(\Omega)$ with $$\Vert u_0 -\widetilde{u}_0 \Vert_{H_{a}^{1}(\Omega)} \leq \varepsilon\,,$$ there exists at least a control function $\widehat{h} \in L^2 (\widehat{\omega} \times (0,T))$ such that the associated state $u$ satisfies
    \begin{equation}\label{ct}
    u(\cdot ,T)= \widetilde{u}(\cdot ,T) \,\,\,\text{in}\,\,\, \Omega_T \,.
    \end{equation}
\end{mydef}

Control of degenerate parabolic equations is a fairly well-developed subject in Control Theory. It is important to remark that semilinear nondegenerate equations have been studied extensively in the last decades, see \cite{dcbz-02, fr-71,clm-12,cz-00,FurImanu-96,lr-95} in the context of bounded cylindrical domains and \cite{cllp-25,lcmml-16,ll-22} in more general domains. Moreover, in the context of degenerate reaction-diffusion equations there are several interesting models, such as models in mathematical biology and in a wide variety of physical situations, see for instance \cite{cpz-04,kalash-87,nss-00}. In recent years several contributions treating degenerate PDEs appeared, in particular we mention the works by Cannarsa and collaborators \cite{Alabau_cannarsa_fragnelli-06,abcl-17,cmv1-04,cmv2-08,cmv3-16,cmv4-17,cty-10}.

The presence of a bilinear control that acts through the reaction term represents a realistic mechanism for influencing the system. In contrast to additive controls, the control considered here acts proportionally to the current state, modulating growth or decay rates. This type of control appears, for example, in chemical reactions where catalysts affect reaction rates, in ecological models where reproduction rates are regulated, or in thermal systems with feedback-dependent dissipation.

From a practical standpoint, the controllability properties studied in this work correspond to the ability to steer the system toward a desired configuration. Null controllability is associated with the suppression or extinction of a physical quantity, while controllability to a positive trajectory reflects the possibility of tracking a prescribed evolution. These objectives are of significant interest in applications where one seeks either to stabilize or to regulate complex dynamical systems under realistic constraints.

The goal of this work is to give some results on the bilinear controllability of a degenerate parabolic system. We refer to the early paper \cite{bms-82} on controllability of an abstract infinite dimensional bilinear system, which appears to be the first work on this subject in the framework of PDEs. In \cite{khapalov-02}, the author discussed the non-negative approximate controllability of a parabolic system with superlinear term governed by a bilinear control. Moreover, in \cite{khapalov_newton-02} he also discussed the bilinear null-controllability of a parabolic system with the reaction term satisfying Newton's Law. We also refer to the article \cite{pzh-06}, on exact controllability of parabolic systems. Important progress has been made recently in the  analysis of bilinear controllability of parabolic equations, we cite, for instance, Alabau-Boussouira et al. \cite{abcu-22,abcu-24} and Cannarsa et al. \cite{cdu-22}. In the context of degenerate hyperbolic equations, we mention Cannarsa et al. \cite{cmu-23}.

The main result in this paper is as follows.
\begin{myth}
    \label{th}
    Under the previous assumptions on the functions $a$ and $f$, the nonlinear system (\ref{eq:PDE}) is locally exactly controllable to the positive trajectory $\widetilde{u}$ at any time $T>0$.
\end{myth}

The strategy to prove Theorem \ref{th} relies on an application of the \emph{Liusternik's Inverse Function Theorem} in Banach spaces; see \cite{Alekseev}. Let us start using a suitable change of variables that transforms (\ref{eq:PDE}) in a parabolic problem in a fixed cylindrical domain. Then, we verify that the assumptions of Liusternik's Theorem are satisfied.
%and consequently for any small initial data $u_0$, (\ref{eq:PDE}) is solvable.

The remainder of the paper is structured as follows. In Section \ref{sec:diffeomorphism}, we give details of the announced change of variables. In Section \ref{sec:linearized_system}, we consider and solve a null controllability problem for an associated linear parabolic equation; this will be needed later to prove that the hypotheses of Liusternik's Theorem are fulfilled. Section \ref{sec:control for nonlinear system} deals with the proof of Theorem \ref{th}. Finally, some additional comments are presented in Section \ref{sec:final_remarks}.

\section{Reduction to a fixed cylindrical domain.}
\label{sec:diffeomorphism}

Using the diffeomorphism  $\tau : Q\rightarrow \widehat{Q}$ the domains are transformed in the following way:
$$
    \Omega \rightarrow \Omega_t, \qquad {\omega} \rightarrow \widehat{\omega},
    %\qquad {\omega}_1 \rightarrow \widehat{\omega}_1
$$
and the functions are transformed as
$$
    \widehat{h}(\overline{x},t)=h(\tau_t(x),t), \qquad \mathbb{1}_{\widehat{\omega}}(\overline{x},t)=\mathbb{1}_{{\omega}}(\tau_t(x),t).
$$

The state function $u$ of our PDE (\ref{eq:PDE}) is transformed in a function $y$ in the cylindrical domain $Q$ such that
$$
    u(\overline{x},t)=y(x,t)=y\left(\tau_t^{-1}(\overline{x}),t\right) = y\left(\psi_t(\overline{x}),t\right),
$$
where we use the notations $\tau_t^{-1}=\psi_t$ and  $\psi(x,t)=\psi_t(x)$.

To obtain the equation verified by $y$ we use the formula
$$
    u_{\overline{x}} = \dpar{u}{\overline{x}}=\dpar{y}{x}\dpar{x}{\overline{x}} = y_x\dpar{\psi}{\overline{x}}(\tau_t(x),t).
$$

Thus,
\begin{equation*}
    \begin{split}
        ({a}(\overline{x})u_{\overline{x}})_{\overline{x}} &=
        \dpar{}{\overline{x}}\left( {a}(\overline{x}) \dpar{y}{x}\dpar{\psi}{\overline{x}} \right) \\
        &={a}(\overline{x})\dpar{\psi}{\overline{x}}\dpar{ }{\overline{x}}\left(\dpar{y}{x} \right)+\dpar{y}{x}\dpar{ }{\overline{x}}\left({a}(\overline{x})\dpar{\psi}{\overline{x}} \right) \\
        &={a}(\overline{x})\dpar{\psi}{\overline{x}}\dpar{ }{x}\left(\dpar{y}{x} \right)\dpar{\psi}{\overline{x}}+\dpar{y}{x}\dpar{ }{\overline{x}}\left({a}(\overline{x})\dpar{\psi}{\overline{x}} \right) \\
        &=\left( \dpar{\psi}{\overline{x}}(\tau_t(x),t) \right)^2{a}(\tau_t(x))y_{xx}+y_x\left({a}'(\tau_t(x))\dpar{\psi}{\overline{x}}+{a}(\tau_t(x))\dpd{\psi}{\overline{x}}(\tau_t(x),t) \right).
    \end{split}
\end{equation*}

On the other hand, 
$$
    \dpar{u}{t}=\dpar{u}{t}+\dpar{u}{x}\dpar{x}{t}=\dpar{y}{t}+\dpar{y}{x}\dpar{\psi}{t}=y_t+\left(\dpar{\psi}{t}(\tau_t(x),t)\right)y_x.
$$

Therefore, applying the diffeomorphism, our equation in the fixed domain $Q$ takes the form:
\begin{equation}\label{ec1p}
%\hspace*{-1.2cm}	
\begin{cases}
		y_t-\left( \dpar{\psi}{\overline{x}}(\tau_t(x),t) \right)^2{a}(\tau_t(x))y_{xx}+\left(\dpar{\psi}{t}(\tau_t(x),t)-{a}'(\tau_t(x))\dpar{\psi}{\overline{x}} \right) y_x \\
        \quad -{a}(\tau_t(x))\dpd{\psi}{\overline{x}}(\tau_t(x),t) y_x + F\left(\ell(t)x,t,y\right)={h}\cara_{_{{\omega}}}y, & \ \text{in} \ \ \ {Q},\\
		y=0, & \  \text{on} \ \ \ {\Sigma},\\
		y(0)=u_0(\tau_0(x))=y_0, & \ \text{in} \ \ \ \Omega.
	\end{cases}
\end{equation}

For our 1D problem, we use the following diffeomorphism given by rescaling,
$$
    x=\tau_t^{-1}(\overline{x})=\psi(\overline{x},t)=\frac{\overline{x}}{\ell(t)},
$$
sending 
$\Omega_t=\{\overline{x}\in\mathbb{R} \ | \ 0<\overline{x} <\ell(t)\}$ to the fixed domain $\Omega=\{x \in \mathbb{R} \ | \ 0< x<1 \}$.

Then, we have
$$
    \dpar{\psi}{\overline{x}}(\tau_t(x),t)=\frac{1}{\ell(t)}, \ \ \ \ \ \dpd{\psi}{\overline{x}}(\tau_t(x),t)=0, \ \ \ \ \ \ \dpar{\psi}{t}(\tau_t(x),t)=-\overline{x}\frac{\ell'(t)}{\ell(t)^2}=
    %-x\ell(t)\frac{\ell'(t)}{\ell(t)^2}=
    -\frac{\ell'(t)}{\ell(t)}x,
$$
and Equation (\ref{ec1p}) becomes
\begin{equation}\label{ec1b}
	%\hspace*{-1.2cm}	
    \begin{cases}
		y_t-\frac{1}{\ell(t)^2}{a}(\tau_t(x))y_{xx}-\frac{1}{\ell(t)}\left(\ell'(t)x+{a}'(\tau_t(x)) \right)y_x+F\left(\ell(t)x,t,y\right)={h}\cara_{_{{\omega}}}y, & \ \ \ \text{in} \ \ \ {Q},\\
		y=0, & \ \ \ \text{on} \ \ \ {\Sigma},\\
		y(0)=u_0(\tau_0(x))=y_0, & \ \ \ \text{in} \ \ \ \Omega.
	\end{cases}
\end{equation}

We have the identity
$\left({a}(\overline{x})y_x\right)_x= {a}(\overline{x})y_{xx}+{a}'(\overline{x})\dpar{\overline{x}}{x}y_x$, 
then
$$
    -\frac{1}{\ell(t)^2}{a}(\tau_t(x))y_{xx}=-\frac{1}{\ell(t)^2}\left({a}(\tau_t(x))y_x\right)_x+\frac{1}{\ell(t)}{a}'(\tau_t(x))y_x.
$$

Thus, (\ref{ec1b}) becomes
\begin{equation}\label{eqprin1}
	\hspace*{-1.2cm}	\begin{cases}
		y_t-\frac{1}{\ell(t)^2}\left({a}(\ell(t)x)y_x\right)_x-\frac{\ell'(t)}{\ell(t)}xy_x+F\left(\ell(t)x,t,y\right)={h}\cara_{_{{\omega}}}y, & \ \ \ \text{in} \ \ \ {Q},\\
		y=0, & \ \ \ \text{on} \ \ \ {\Sigma},\\
		y(0)=v_0(\tau_0(x))=y_0, & \ \ \ \text{in} \ \ \ \Omega.
	\end{cases}
\end{equation}

Using hypothesis (\ref{eq:cond_a}) for the function $a$, we define the functions 
$$
    b(t)=\frac{g(t)}{\ell(t)^2}, \qquad
    %$b(t)=\frac{a(\ell(t))}{\ell(t)^2}$,
    B(x,t)=\frac{\ell'(t)x}{\ell(t)\sqrt{a}},
$$  
%\ \ \ \ C(t)=\frac{\sqrt{a(\ell(t))}}{\ell(t)}$ 
and write (\ref{eqprin1}) as
\begin{equation}\label{eqprin}
	\hspace*{-1.2cm}	\begin{cases}
		y_t-b(t)\left({a}(x)y_x\right)_x-B(x,t)\sqrt{a}y_x+F\left(\ell(t)x,t,y\right)={h}\cara_{_{{\omega}}}y, & \ \ \ \text{in} \ \ \ {Q},\\
		y=0, & \ \ \ \text{on} \ \ \ {\Sigma},\\
		y(0)=u_0(\tau_0(x))=y_0, & \ \ \ \text{in} \ \ \ \Omega.
	\end{cases}
\end{equation}

%\color{red}
By hypothesis \textbf{H3} we remark that $b'(t)/b(t)$ and $B(x,t)$ are bounded functions.
%\color{black}

The uncontrolled trajectory equation after the change of variable becomes
\begin{equation}\label{eqprinsin}
	%\hspace*{-1.2cm}	
    \begin{cases}
		\widetilde{y}_t-b(t)\left({a}(x)\widetilde{y}_x\right)_x-B(x,t)\sqrt{a}\widetilde{y}_x+F\left(\ell(t)x,t,\widetilde{y}\right)=0, & \ \ \ \text{in} \ \ \ {Q},\\
		\widetilde{y}=0, & \ \ \ \text{on} \ \ \ {\Sigma},\\
		\widetilde{y}(0)=\widetilde{y}_0, & \ \ \ \text{in} \ \ \ \Omega.
	\end{cases}
\end{equation}

Substracting \eqref{eqprinsin} from  \eqref{eqprin} and making the change of variables $z=y-\widetilde{y}$ we get the system
\begin{equation}\label{nolineal}
    \begin{cases}
		z_t-b(t)\left({a}(x)z_x\right)_x-B(x,t)\sqrt{a}z_x+F\left(\ell(t)x,t,z+\widetilde{y}\right)-F\left(\ell(t)x,t,\widetilde{y}\right)={h}\cara_{_{{\omega}}}(z+\widetilde{y}), & \ \ \ \text{in} \ \ \ {Q},\\
		z=0, & \ \ \ \text{on} \ \ \ {\Sigma},\\
		z(0)=y_0-\widetilde{y}_0=z_0, & \ \ \ \text{in} \ \ \ \Omega.
    \end{cases}
\end{equation}

\section{Analysis of the controllability of the linearized system.}
\label{sec:linearized_system}

Consider the linearization of \eqref{nolineal} at zero,
\begin{equation}\label{lineal}
    \begin{cases}
		z_t-b(t)\left({a}(x)z_x\right)_x-B(x,t)\sqrt{a}z_x+D_3F\left(\ell(t)x,t,\widetilde{y}\right)z=\tilde h \cara_{_{{\omega}_1}} + G, & \ \ \ \text{in} \ \ \ {Q},\\
		z=0, & \ \ \ \text{on} \ \ \ {\Sigma},\\
		z(0)=y_0-\widetilde{y}_0=z_0, & \ \ \ \text{in} \ \ \ \Omega.
    \end{cases}
\end{equation}

In equation 
(\ref{lineal}) we define the new control $\tilde h$ such that
$$
    \tilde h \cara_{_{{\omega}_1}}:={h}\cara_{_{{\omega}_1}}\widetilde{y}, \qquad \text{where} \quad \omega_1\subset\subset \omega.
$$

As usual, the controllability of (\ref{lineal}) is closely related to the properties of the associated adjoint states. In this case, the adjoint of (\ref{lineal}) is given by

\begin{equation}\label{linealadjunto1}
	\begin{cases}
		-\varphi_t-b(t)\left({a}(x)\varphi_x\right)_x+\left(B(x,t)\sqrt{a}\varphi\right)_x+D_3F\left(\ell(t)x,t,\widetilde{y}\right)\varphi=H, & \ \ \ \text{in} \ \ \ {Q},\\
		\varphi=0, & \ \ \ \text{on} \ \ \ {\Sigma},\\
		\varphi(T)=\varphi^T, & \ \ \ \text{in} \ \ \ \Omega,
	\end{cases}
\end{equation}
where $\varphi^T \in L^2(\Omega)$ and $H \in L^2(Q)$.

From the definition of $B(x,t)$ we have $\left(B(x,t)\sqrt{a}\varphi\right)_x=B(x,t)\sqrt{a}\varphi_x+\frac{\ell'(t)}{\ell(t)}\varphi$ and equation  (\ref{linealadjunto1}) becomes: 
\begin{equation}\label{linealadjunto}
	\begin{cases}
		-\varphi_t-b(t)\left({a}(x)\varphi_x\right)_x+B(x,t)\sqrt{a}\varphi_x+c(x,t)\varphi=H, & \ \ \ \text{in} \ \ \ {Q},\\
		\varphi=0, & \ \ \ \text{on} \ \ \ {\Sigma},\\
		\varphi(T)=\varphi^T, & \ \ \ \text{in} \ \ \ \Omega,
	\end{cases}
\end{equation}
where $c(x,t)=\frac{\ell'(t)}{\ell(t)}+D_3F\left(\ell(t)x,t,\widetilde{y}\right)$.

\subsection{Hilbert spaces in the divergence case}

Following \cite{Alabau_cannarsa_fragnelli-06} in the \emph{weakly degenerate} case, for a system in divergence form, we consider the weighted Hilbert spaces:
\begin{equation*}
    \begin{split}
        H^1_a(0,1)=&\left\{ u\in L^2(0,1) \ : \ u \ \text{is absolutely continuous in} \ \ (0,1], \right. \\
        &\left. \quad \sqrt{a}u_x\in L^2(0,1) \ \text{and} \  u(0)=u(1)=0 \right\}        
    \end{split}
\end{equation*}
and 
$$
    H^2_a(0,1)=\left\{u\in H^1_a(0,1) \ : \ au_x\in H^1(0,1) \right\}.
$$

In both cases, we consider inner products and norms given by
$$
    \langle u,v \rangle_{H^1_a(0,1)}=\langle u,v \rangle_{L^2(0,1)}+\langle \sqrt{a}u_x,\sqrt{a}v_x \rangle_{L^2(0,1)}, \ \ \|u\|^2_{H^1_a(0,1)}=\|u\|^2_{L^2(0,1)}+\|\sqrt{a} u_x\|^2_{L^2(0,1)}
$$
for all $u, v\in H^1_a(0,1)$, and
$$
    \langle u,v \rangle_{H^2_a(0,1)}=\langle u,v \rangle_{H^1_a(0,1)}+\langle ({a}u_x)_x,({a}v_x)_x \rangle_{L^2(0,1)}, \ \ \|u\|^2_{H^2_a(0,1)}=\|u\|^2_{H^1_a(0,1)}+\|({a}u_x)_x\|^2_{L^2(0,1)}
$$
for all $u, v\in H^2_a(0,1)$.
   
\subsection{Carleman estimates}

Let $\omega'=(\alpha',\beta') \subset\subset \omega$ and let $\Psi : [0,1] \to \R$ be a $C^2$ function such that
$$
    \Psi(x) = 
    \begin{cases}
        \int_0^x \frac{s}{a(s)} ds, \quad x \in [0,\alpha'), \\
        -\int_{\beta'}^x \frac{s}{a(s)} ds, \quad x \in [\beta',1].
    \end{cases}
$$

For $\lambda \geq \lambda_0$ define the functions
\begin{equation*}
    \begin{split}
        &\theta(t) = \frac{1}{(t(T-t))^4}, \qquad \eta(x) = e^{\lambda(|\Psi|_\infty + \Psi)}, \qquad \sigma(x,t) = \theta(t) \eta(x), \\
        &\qquad\qquad\text{and}\quad \varphi(x,t) = \theta(t) (e^{\lambda(|\Psi|_\infty + \Psi)}-e^{3\lambda|\Psi|_\infty}).
    \end{split}    
\end{equation*}

The following proposition was proved for a cylindrical domain in \cite{DemarqueLimacoViana_deg_eq2018}
and for a non-cylindrical domain in \cite{GYL-Carleman-2025}.
Indeed, we need a Carleman estimate for the linear equation
\begin{equation}\label{adjunto2_texto_art}
	\begin{cases}
		v_t-b(t)\left(a(x)v_x\right)_x-d(x,t)\sqrt{a}v_x+c(x,t)v =H(x,t), & \ \ \ \text{in} \ \ \ {Q},\\
		v=0, & \ \ \ \text{on} \ \ \ {\Sigma},\\
        v(T) = v^T,& \ \ \ \text{in} \ \ \ (0,1),
	\end{cases}
\end{equation}
where $b$ is a continuous positive function in $[0,T]$ such that  $\frac{b'(t)}{b(t)} \leq C_b$, for some constant $C_b>0$, 
and we denote
$$
    m := \min\limits_{t \in [0,T]} b(t)>0, \quad M :=\max\limits_{t \in [0,T]} b(t).
$$

\begin{propo}[\cite{GYL-Carleman-2025}]
\label{prop_carleman_1}
    There exist $C>0$ and $\lambda_0,  s_0>0$ such that every solution $v$ of (\ref{adjunto2_texto_art}) satisfies, for all
    $s\geq s_0$, $\lambda\geq \lambda_0$ and any $v^T\in L^2(Q)$, that
    \begin{equation}\label{ecjv34}
        \int_{0}^{T}\int_{0}^{1}e^{2s\varphi}\left(  (s\lambda)\sigma a b^2 v_x^2+(s\lambda)^2\sigma^2b^2 v^2 \right)\leq C\left(	\int_{0}^{T}\int_{0}^{1} e^{2s\varphi} |H|^2+(\lambda s)^3\int_{0}^{T}\int_{\omega_1} e^{2s\varphi} \sigma^3  v^2  \right).
    \end{equation}
\end{propo}

In order to obtain the global null controllability of the linearized system, we need a Carleman inequality with weights that do not vanish at $t=0$. For that, consider the function $m \in C^\infty([0,T])$ satisfying $m(0)>0$,
\begin{equation}
    \label{eq:def_m}
    m(t) \geq t^4(T-t)^4, \quad t \in (0,T/2], \qquad\qquad m(t) = t^4(T-t)^4, \quad t \in [T/2,T],  
\end{equation}
and define
$$
    \tau(t) = \frac{1}{m(t)}, \qquad \zeta(x,t) = \tau(t) \eta(x), \qquad A(x,t)  = \tau(t) (e^{\lambda(|\Psi|_\infty + \Psi)}-e^{3\lambda|\Psi|_\infty}).
$$
\begin{propo}[Carleman Estimate \cite{GYL-Carleman-2025}]\label{prop:carl_weights_t} 
    There exist positive constants $C$, $\lambda_0$ and $s_0$ such that, for any $s \geq s_0$, $\lambda \geq \lambda_0$ and any $v^T\in L^2(Q)$, the corresponding solution $v$ of (\ref{adjunto2_texto_art}) satisfies
    $$
        \int_{0}^{T}\int_{0}^{1}e^{2sA}\left(  (s\lambda)\zeta a b^2 v_x^2+(s\lambda)^2\zeta^2b^2 v^2 \right) \leq C \left( \int_Q e^{2s A} |H|^2dxdt + \int_{\omega_1 \times(0,T)} e^{2s A} s^3 \lambda^3 \zeta^3 |v|^2 dxdt\right).
    $$
\end{propo}

%\begin{proof}
%    The proof of this proposition is given in \cite{GYL-Carleman-2025}. 
%\end{proof}

As a corollary, we get the following observability inequality.
\begin{coro} 
    \label{cor:observability}
    There exist positive constants $C$, $\lambda_0$ and $s_0$ such that, for any $s \geq s_0$, $\lambda \geq \lambda_0$ and any $v^T\in L^2(Q)$, the corresponding solution $v$ of (\ref{adjunto2_texto_art}) with $H = 0$, satisfies
    \begin{equation}\label{Observability1}
        \|v(0)\|^2_{L^2(0,1)}  \leq C \int_{0}^{1}\int_{\omega_1} e^{2s A} s^3 \lambda^3 \zeta^3 |v|^2dxdt.
    \end{equation}	
\end{coro}

{\bf A null controllability result for the linear system}

The last goal of this section is to establish a result of global null controllability for the linear problem
\begin{equation}\label{eq:linearized_system1}
	\begin{cases}
		z_t-b(t)\left({a}(x)z_x\right)_x-B(x,t)\sqrt{a}z_x+D_3F\left(\ell(t)x,t,\widetilde{y}\right)z=\tilde h \cara_{_{{\omega}_1}}+G, & \ \ \ \text{in} \ \ \ {Q}\\
		z=0, & \ \ \ \text{on} \ \ \ {\Sigma}\\
		z(0)=y_0-\widetilde{y}_0=z_0, & \ \ \ \text{in} \ \ \ \Omega
	\end{cases}
\end{equation}
where $G \in L^2((0,1)\times (0,T))$, $\tilde h\in L^2(\omega_1\times (0,T))$ and $a$ satisfy assumption \textbf{H1}.

%\color{red}
In the following estimates, we need weights that depend only on $t$. Thus, we define
\begin{equation*}
    \begin{cases}
        A^*(t) = \displaystyle\max_{x \in (0,1)} A(x,t), \qquad \hat A(t) = \displaystyle\min_{x \in (0,1)} A(x,t), \\
        \zeta^*(t) = \displaystyle\max_{x \in (0,1)} \zeta(x,t), \qquad \hat \zeta(t) = \displaystyle\min_{x \in (0,1)} \zeta(x,t),
    \end{cases}
\end{equation*}
and we observe that $A^*(t) < 0$, $\hat A(t) < 0$ and that $\zeta^*(t) / \hat \zeta(t)$ does not depend on $t$ and is equal to some constant $\zeta_0 \in \R$. Moreover, if $\lambda$ is sufficiently large, we can suppose
\begin{equation}
    \label{eq:comp_pesos}
    3 A^*(t) < 2 \hat A(t) < 0.
\end{equation}

Let us define 
$$
    \hat \Gamma(\phi) = \int_Q e^{2s \hat A} \left[(s\lambda) \hat \zeta b^2 a |\phi_x|^2 + (s\lambda)^2 {\hat \zeta}^2 b^2 |\phi|^2 \right] dx dt.
$$

Thus, Proposition \ref{prop:carl_weights_t} and Corollary \ref{cor:observability}  imply the following corollary where the weights depend only on $t$.

%\color{blue}
\begin{coro} 
\label{cor:carleman_pesos_t}
    There exist positive constants $C$, $\lambda_0$ and $s_0$ such that, for any $s \geq s_0$, $\lambda \geq \lambda_0$ and any $v^T \in L^2(Q)$, the corresponding solution $\phi$ of \eqref{adjunto2_texto_art} satisfies
        \begin{equation}\label{pesos_t_Carleman for eq:adjoint_optimality_system}
        \begin{split}
            \|v(0)\|^2_{L^2(0,1)} +
            \hat \Gamma(v) \leq & C \left( \int_Q e^{2s A^*} (\zeta^*)^4 |H|^2 dxdt + \int_{\omega_1 \times(0,T)} e^{2s A^*} (\zeta^*)^8 |v|^2 dxdt\right).
        \end{split}
    \end{equation}

    \begin{equation}\label{pesos_t_Carleman for eq:adjoint_optimality_system2}
	\begin{split}
		\|v(0)\|^2_{L^2(0,1)} +
		\hat \Gamma(v) \leq & C \left( \int_Q e^{2s A^*}  |H|^2 dxdt + \int_{\omega_1 \times(0,T)} e^{2s A^*} (\zeta^*)^3 |v|^2 dxdt\right).
	\end{split}
\end{equation}
\end{coro}
%\color{red}

Let us define the weights:
%-------------------------------------------------------
\begin{equation}\label{eq:weights_rhos}
	\left\{ \begin{array}{l}
		\rho_0 = e^{-sA^*} , \qquad   \rho_1 = e^{-sA^*}  (\zeta^*)^{-3/2},\\  \rho_2 = e^{-3sA^*/2}  \hat \zeta^{-1},  \qquad \hat{\rho} = e^{-sA^*} (\zeta^*)^{-3/4}, \\
        \overline{\rho}=e^{-sA^*}  (\zeta^*)^{-5/2},
	\end{array}\right.
\end{equation}
which satisfy
\begin{equation}\label{eq:compara_rhos}
     \overline \rho \leq C\rho_{1}\leq C \hat{\rho}\leq C\rho_{0}\leq C\rho_{2}, \qquad \hat{\rho}^{2}=\rho_{1}\rho_{0} \qquad \text{and} \qquad \rho_2 \leq C \rho_1^2.   
\end{equation}

%\color{blue}
In particular, Corollary \ref{cor:carleman_pesos_t} and \eqref{eq:comp_pesos} imply 
\begin{equation}
    \label{eq:carleman_simples}
    \begin{split}
        \|v(0)\|^2_{L^2(0,1)} +
        \int_Q \rho_2^{-2} |v|^2 dxdt
        \leq & \ C \left( \int_Q \rho_0^{-2} |H|^2 dxdt 
        + \int_{\omega_1 \times(0,T)} \rho_1^{-2} |v|^2 dxdt\right).
    \end{split}
\end{equation}
%\color{red}

In the following theorem we show the global null controllability of the linearized system \eqref{eq:linearized_system1}. In particular, since the weight $\rho_0$ blows up at $t=T$, \eqref{estimate for solution} shows that $z(\cdot,T)=0$ in $[0,1]$.

\begin{teo}%\label{theorem case linear}
\label{prop:linear_control}
    If $z_{0}\in L^{2}(0,1)$, $\rho_2 G \in L^2(Q)$, then there exists a control $\tilde h\in L^{2}(\omega_1 \times (0,T))$ with associated state $z \in C^{0}([0,T];L^{2}(0,1))\cap L^{2}(0,T;H^{1}_{a}(0,1))$, solution of \eqref{eq:linearized_system1}, such that
    \begin{equation}\label{estimate for solution}
        \int_Q \rho_0^2 |z|^2 dxdt + \int_{\omega_1 \times (0,T)} \rho_1^2 |\tilde h|^2dxdt \leq C \kappa_{0}(G,z_{0}),   
    \end{equation}
    where $\kappa_{0}(G,z_{0})= \|\rho_2 G\|^2_{L^2(Q)} + \|z_0\|^2_{L_2(0,1)}$. In particular, $z(x,T)=0$, for all $x\in [0,1]$. 
\end{teo}

\begin{proof}
Let us denote by
$$
    L\varphi = \varphi_t-b(t)\left({a}(x) \varphi_x\right)_x -B(x,t)\sqrt{a} \varphi_x+D_3F\left(\ell(t)x,t,\widetilde{y}\right)\varphi.
$$
Then, we define 
\begin{equation*}
    \begin{array}{l}
        \mathcal{P}_{0} = \{ \phi \in C^{2}(\overline{Q})^{3} \ : \ \phi(0,t)=\phi(1,t) = 0\ \text{a.e in}\ (0,T)\}
    \end{array}
\end{equation*}
and the application $b:\mathcal{P}_{0}\times \mathcal{P}_{0}\rightarrow \mathbb{R}$ given by
\begin{equation}
\label{eq:def_b}
    \begin{array}{l}
          b(\tilde{\phi},\phi) = \displaystyle\int_{Q}\rho_{0}^{-2}(L^{\ast}\tilde{\phi})(L^{\ast}\phi)\ dx\ dt + \displaystyle\int_{\omega_1 \times (0,T)}\rho_{1}^{-2}\tilde{\phi}\phi\ dx\ dt,\,\,\ \ \forall \phi,\tilde{\phi} \in \mathcal{P}_{0},
    \end{array}
\end{equation}
which is bilinear on $\mathcal{P}_{0}$ and defines an inner product. Indeed, taking $\tilde{\phi}=\phi$ in \eqref{eq:def_b}, we have that, by \eqref{eq:carleman_simples}  $b(\cdot,\cdot)$ is positive definite. The other properties are straightforwardly verified.

Let us consider the space $\mathcal{P}$ the completion of $\mathcal{P}_{0}$ for the norm associated to $b(\cdot,\cdot)$ (which we denote by $\|.\|_{\mathcal{P}}$). Then, $b(\cdot,\cdot)$ is symmetric, continuous and coercive bilinear form on $\mathcal{P}$.

Now, let us define the functional linear $\ell :\mathcal{P}\rightarrow\mathbb{R}$ as
\begin{equation*}
    \langle\ell, \phi\rangle = \displaystyle\int_{0}^{1}y_{0}\phi(0)dx + \displaystyle\int_{Q}(H\phi)dx dt.
\end{equation*}

Note that $\ell$ is a bounded linear form on $\mathcal{P}$. Indeed, applying the classical Cauchy-Schwartz inequality 
and using \eqref{eq:carleman_simples}, we get
\begin{equation}\label{l limitado}
    \begin{array}{l}
           |\langle\ell, \phi\rangle| \leq |y_{0}|_{L^{2}(0,1)}|\phi(0)|_{L^{2}(0,1)} + |\rho_{2}H|_{L^{2}(Q)}|\rho_{2}^{-1}\phi|_{L^{2}(Q)} \\
           \leq C(|y_{0}|^{2}_{L^{2}(0,1)} + |\rho_{2}H|^{2}_{L^{2}(Q)} )^{1/2}\left(b(\phi,\phi)\right)^{1/2}\\
         \leq C(|y_{0}|^{2}_{L^{2}(0,1)} + |\rho_{2}H|^{2}_{L^{2}(Q)} )^{1/2}\| \phi\|_{\mathcal{P}},
    \end{array}
\end{equation}
for all $\phi \in\mathcal{P}$. Consequently, in view of Lax-Milgram's theorem, there is only one $\hat{\phi} \in \mathcal{P}$ satisfying
\begin{equation}\label{eq: por Lax-M.}
    b(\hat{\phi},\phi) = \langle\ell, \phi\rangle,\,\,\ \ \forall \phi \in\mathcal{P}.
\end{equation}

Let us set
\begin{equation}\label{definição de y, pi, h}
    \left\{\begin{array}{lll}
          z = \rho_{0}^{-2}(L^{\ast}\hat{\phi})&\text{in} & Q,\\
          %p^{i} = \rho_{0}^{-2}(L\hat{\psi^{i}} + \dfrac{1}{\mu_{i}}\hat{\phi}1_{O_{i}}),\, i=\{1,2\} &\text{in} & Q,\\
          \tilde h = -\rho_{1}^{-2}\hat{\phi}1_{\omega_1} &\text{in} & Q.
    \end{array}\right.
\end{equation}

Then, replacing \eqref{definição de y, pi, h} in \eqref{eq: por Lax-M.} we have
\begin{equation*}
    \begin{array}{l}
         \displaystyle\int_{Q} z B\,dx\,dt 
         = \displaystyle\int_{0}^{1}z_{0}\phi(0)dx + \displaystyle\int_{\omega_1 \times(0,T)} \tilde h \phi\,dx\,dt+\displaystyle\int_{Q}(H\phi)dx\ dt,
    \end{array}
\end{equation*}
where $\phi$ is a solution of the system
\begin{equation*}
    \left\{\begin{array}{lll}
       L^{\ast}\phi = B   &\text{in}& Q,  \\
       \phi(0,t)=\phi(1,t)=0 & \text{on} & (0,T), \\
       \phi(\cdot,T)=0,\ &\text{in}& \Omega.
   \end{array}\right.
\end{equation*}

Therefore, $z$ is a solution by transposition of \eqref{eq:linearized_system1}. Also, since $\hat{\phi} \in \mathcal{P}$ and 
$H \in L^{2}(Q)$, using, for instance, the well-posedness result of Appendix A in \cite{GYL-Carleman-2025}
applied to a linear equation, we obtain
$$
    z \in C^{0}([0,T];L^{2}(0,1))\cap L^{2}(0,T;H^{1}_{a}(0,1)).
$$

Moreover, from \eqref{l limitado} and \eqref{eq: por Lax-M.}
\begin{equation*}
    \left(b(\hat{\phi},\hat{\phi})\right)^{1/2}  \leq C(|z_{0}|^{2}_{L^{2}(0,1)} + |\rho_{2}G|^{2}_{L^{2}(Q)})^{1/2},
\end{equation*}
that is,
\begin{equation*}
    \begin{array}{l}
        \displaystyle\int_{Q}\rho_{0}^{2}|y|^{2}\ dxdt    +\displaystyle\int_{\omega_1 \times (0,T)}\rho_{1}^{2}|\tilde h|^{2} dx dt
        \leq C(\|z_{0}\|^{2}_{L^{2}(0,1)} + \|\rho_{2}H\|^{2}_{L^{2}(Q)} ),
    \end{array}
\end{equation*}
proving \eqref{estimate for solution}.

\end{proof}

%\color{black} 
The next result, stated below, is a consequence of Theorem \ref{prop:linear_control}. Arguing in a similar way, we will establish the regularity of the control $h$, which will be very useful to work out estimates involving multiplicative control in Section 4. The aforementioned regularity is important to control the problem (\ref{nolineal}).

\begin{propo}
Let $\overline{\rho}=e^{-sA^*}  (\zeta^*)^{-5/2}$. Then, one has
\begin{equation}
		\label{cr1}
		\overline{\rho} \tilde{h} \in \mathsf{U}:=L^{2}(0,T;H^{1}_{0}(0,1) \cap H^{2}(0,1)) \cap L^{\infty}(0,T;H^{1}_{0}(0,1)) \qquad and
	\end{equation}
	\vspace{-0.5cm}
	\begin{equation}
		\label{cr2}
		\!\!\!\!\!\!\!\!\!\!\!\!\!\!\!\!\!\!\!\!\!\!\!\Vert \overline{\rho} \tilde{h}\|_{\mathsf{U}}\leq C\left(\Vert z_{0}\Vert_{L^2(0,1)} +\Vert \rho_2 G\Vert_{L^2(Q)} \right),
	\end{equation}
	where $$\Vert \cdot \Vert_{\mathsf{U}} =\Vert \cdot \Vert_{L^{2}(0,T;H^{1}_{0}(0,1) \cap H^{2}(0,1))}+\Vert \cdot \Vert_{L^{\infty}(0,T;H^{1}_{0}(0,1))}$$
\end{propo}
\begin{proof}
 Let us set $\widetilde{\rho}:=e^{sA^*}  (\zeta^*)^{1/2}$. Proceeding as in the proof of Theorem \ref{prop:linear_control} for $\overline{\hat{\phi}}:=\widetilde{\rho} \hat{\phi}$ in \eqref{definição de y, pi, h}  one has
\begin{equation}\label{lag.1}
	\left\{\begin{array}{ll}
		- (\widetilde{\rho} \hat{\phi})_{t} - b(t)\left({a}(x)(\widetilde{\rho} \hat{\phi})_x\right)_x+ \left(B(x,t)\sqrt{a}(\widetilde{\rho} \hat{\phi})\right)_x+D_3F\left(\ell(t)x,t,\widetilde{y}\right)(\widetilde{\rho} \hat{\phi}) & \  \\
		
		 \quad	= - \widetilde{\rho}\rho_{0}^{2} z - \widetilde{\rho}_{t} \hat{\phi}  & \text{in} \quad Q,\\
		\widetilde{\rho}\hat{\phi}  = 0     ,                             & \text{on} \quad \Sigma,\\
		\widetilde{\rho} \hat{\phi}(T) =0    ,                                  & \text{in} \quad \Omega , \\
		\widetilde{\rho}\hat{\phi} = - \widetilde{\rho}\rho_{1}^{2} \tilde{h}    & \text{in} \quad (x,t)\times \omega_1.
	\end{array}\right.
\end{equation}

It is not difficult to check that $\vert \widetilde{\rho}\rho_{0}^{2} z\vert \leq C \vert \rho_0 z\vert$ and $\vert \widetilde{\rho}_{t} \hat{\phi}\vert \leq C\vert \rho_{0}^{-1} \hat{\phi}\vert$.

So, from Theorem \ref{prop:linear_control}, one has 
$$\Vert \widetilde{\rho}\rho_{0}^{2} z \Vert_{L^2 (Q)}^2 \leq C\left( \Vert z_0\Vert_{L^2 (0,1)}^{2}+\Vert \rho_{2} G \Vert_{L^2(Q)}^{2} \right)$$
and from Carleman estimate, we can deduce
$$\Vert \widetilde{\rho}_{t} \hat{\phi} \Vert_{L^2 (Q)}^2 \leq C\left( \Vert z_0\Vert_{L^2 (0,1)}^{2}+\Vert \rho_{2} G \Vert_{L^2(Q)}^{2} \right)\,.$$
Therefore,
$$(\widetilde{\rho}\rho_{0}^{2} z,  \widetilde{\rho}_{t} \hat{\phi}) \in [L^2 (Q)]^2 \quad \text{and}$$
\begin{equation}\label{lag.2}
 \Vert \widetilde{\rho}\rho_{0}^{2} z \Vert_{L^2 (Q)}^2 + \Vert \widetilde{\rho}_{t} \hat{\phi} \Vert_{L^2 (Q)}^2 \leq C\left( \Vert z_0\Vert_{L^2 (0,1)}^{2}+\Vert \rho_{2} G \Vert_{L^2(Q)}^{2} \right)\,.
 \end{equation}

 In view of \eqref{lag.1}, \eqref{lag.2} and parabolic regularity, we have that
 $$\widetilde{\rho}\hat{\phi} \in L^{2}(0,T;H^{1}_{0}(0,1) \cap H^{2}(0,1)) \cap L^{\infty}(0,T;H^{1}_{0}(0,1))\,.$$

 Now, from \eqref{lag.1} one has
 \begin{equation}\label{rc}
 \begin{split}
    \widetilde{\rho}\hat{\phi} =& - \widetilde{\rho}\rho_{1}^{2} \tilde{h}   \\
        =& -(e^{sA^*}  (\zeta^*)^{1/2})\cdot (e^{-2sA^*}  (\zeta^*)^{-3}) \tilde{h}\\
        =& (e^{-sA^*}  (\zeta^*)^{-5/2}) \tilde{h}\\
    =&  -\overline{\rho} \tilde{h} .  
 \end{split}
\end{equation}

 Therefore,
 $$\overline{\rho} \tilde{h} \in L^{2}(0,T;H^{1}_{0}(0,1) \cap H^{2}(0,1)) \cap L^{\infty}(0,T;H^{1}_{0}(0,1))$$
 and
 $$\Vert \overline{\rho} \tilde{h} \Vert_{L^2(0,T; H^1_0 (0,1) \cap H^2 (0,1))} + \Vert \overline{\rho} \tilde{h}  \Vert_{L^\infty(0,T; H^1_0(0,1))} \leq C\left( \Vert z_0\Vert_{L^2 (0,1)}^{2}+\Vert \rho_{2} G \Vert_{L^2(Q)}^{2} \right)\,. $$
\end{proof}

\noindent{\bf Additional estimates}

In order to get the local null controllability of the nonlinear equation, we need the following additional estimates.

\begin{propo}
    \label{addicional_estimates_case_linear}
    Under the hypothesis of Proposition \ref{prop:linear_control}, we have, furthermore, that the control $\tilde h\in L^{2}(\omega_1 \times (0,T))$ and the associated states $z \in C^{0}([0,T];L^{2}(0,1))\cap L^{2}(0,T;H^{1}_{a}(0,1))$, solution of \eqref{eq:linearized_system1}, satisfies the additional estimates 
    \begin{equation}\label{des Proposition 5}
        \begin{array}{c}
            \displaystyle\sup_{t \in [0,T]}(\hat{\rho}^{2}\|z\|^{2}_{L^{2}(0,1)})        +\displaystyle\int_{Q}\hat{\rho}^{2} a(x)(|z_{x}|^{2} )dxdt\ \leq C \kappa_{0}(G,z_{0}),  
        \end{array}
    \end{equation}
    where $\kappa_{0}(G,z_{0})= \|\rho_2 G\|^2_{L^2(Q)} + \|z_0\|^2_{L_2(0,1)}$.
    Moreover, if $z_{0}\in H^{1}_{a}(0,1)$ 
    \begin{equation}\label{des Proposition 6}
        \begin{array}{c}
            \displaystyle\sup_{[0,T]}(\rho_{1}^{2}\|\sqrt{a}z_{x} \|^{2}_{L^{2}(0,1)})
            + \displaystyle\int_{Q}\rho_{1}^{2}(|z_{t}|^{2} + |(a(x)z_{x})_{x}|^{2} )dxdt
            \leq C \kappa_{1}(G,z_{0}),  
        \end{array}
    \end{equation}
    where $\kappa_{1}(G,z_{0})= \|\rho_2 G\|^2_{L^2(Q)} + \|z_0\|^2_{H^{1}_{a}(0,1)}$. 
\end{propo}
%\color{red}
\begin{proof}
We proceed following the steps of \cite{DemarqueLimacoViana_deg_eq2018}.
We consider more generally the equation  
\begin{equation}
    \label{eq:linearized_system new}
    \left\{\begin{aligned}
        &y_t - b(t)\left(a(x) y_x\right)_x  
        + c_1(x,t)y + d_1(x,t)\sqrt{a}y_x = \tilde h 1_{\omega_1} + G & \text{in } Q, \\
        &y(0,t)=y(1,t)=0 & \text{on } (0,T), \\
        &y(\cdot,0) = y^0 & \text{in } \Omega,
        \end{aligned}
    \right.
\end{equation}
where the coefficients $b(t)$, %$=\frac{1}{\ell(t)^2}$, 
$c_i$, $d_i$, $i=1,2$ are bounded and $b(t)$ is bounded away from zero.

Let us multiply $\eqref{eq:linearized_system new}$ by $\hat{\rho}^{2} y$ and integrate over $[0,1]$. Hence, using that $\hat{\rho}^{2} = \rho_{0}\rho_{1}$, and $\rho_{1}\leq C\rho_{2}$, we compute
\begin{equation}\label{second estimate}
    \begin{array}{l}
        \dfrac{1}{2}\dfrac{d}{dt}\displaystyle\int_{0}^{1}\hat{\rho}^{2}|y|^{2}dx + b(t) \displaystyle\int_{0}^{1}\hat{\rho}^{2}a(x)|y_{x}|^{2} dx 
        + \displaystyle\int_0^1 \hat{\rho} d_1(x,t) \sqrt{a(x)}y_x y dx   \\
        \leq C\left(\displaystyle\int_{0}^{1}\rho_{0}^{2}|y|^{2}dx + \displaystyle\int_{\omega_1}\rho_{1}^{2}|\tilde h|^{2}dx + \displaystyle\int_{0}^{1}\rho_{2}^{2}|G|^{2}dx\right) + \mathcal{M},
    \end{array}
\end{equation}
where $\mathcal{M} = \displaystyle\int_{0}^{1}\hat{\rho}(\hat{\rho})_{t}|y|^{2}dx$ 
and $(\cdot)_{t}=\frac{d}{dt}(\cdot)$.
Recall that
$A^*(t) = C_1 \tau(t)$, and $\zeta^*(t) = C_2 \tau(t)$, then we have that $(A^*)_{t}=\bar C (\zeta^*)_{t}$ and consequently
\begin{equation*}
	\begin{split}
		\hat{\rho}(\hat{\rho})_{t} &= e^{-sA^*}(\zeta^*)^{-3/4}\left(-s(A^*)_{t} e^{-sA^*} (\zeta^*)^{-3/4} - \frac{3}{4} e^{-sA^*}(\zeta^*)^{-7/4}(\zeta^*)_{t}\right)    \\
		&= -e^{-2sA^*}(\zeta^*)_{t}\left(s(\zeta^*)^{-3/2}\bar{C} + \frac{3}{4}(\zeta^*)^{-5/2} \right)  \\
		&=-\rho_{0}^{2}(\zeta^*)_{t}\left(s(\zeta^*)^{-3/2}\bar{C} + \frac{3}{4}(\zeta^*)^{-5/2} \right).
	\end{split}
\end{equation*}

Thus, since $\tau_t \leq C_3 \tau^{5/4}$ and $\tau^{-1}$ is bounded, then, for any $t\in [0,T)$,
\begin{equation*}
    \begin{array}{l}
        |\hat{\rho}(\hat{\rho})_{t}|\leq C_4\rho_{0}^{2}\tau^{5/4}|s \tau ^{-3/2}\bar{C} + \frac{3}{4} \tau^{-5/2}|  
         \leq C_4\rho_{0}^{2}|s\bar{C} \tau^{-1/4} + \frac{3}{4} \tau^{-5/4}| \leq C_5\rho_{0}^{2},
    \end{array}
\end{equation*}
and we get
\begin{equation*}
%\label{estimate of M}
    \mathcal{M}\leq C_5 \displaystyle\int_{0}^{1}{\rho_{0}^{2}}|y|^{2} dx.
\end{equation*}

Thus, using Young's inequality, \eqref{second estimate} becomes, for a small $\epsilon>0$, 
\begin{equation*}
    \begin{array}{l}
        \dfrac{1}{2}\dfrac{d}{dt}\displaystyle\int_{0}^{1}\hat{\rho}^{2}(|y|^{2})dx + b(t) \displaystyle\int_{0}^{1}\hat{\rho}^{2} a(x)(|y_{x}|^{2})dx \\%\vspace{0.1cm}\\
        \leq C\left(\displaystyle\int_{0}^{1}\rho_{0}^{2}(|y|^{2} )dx + \displaystyle\int_{\omega_1}\rho_{1}^{2}|\tilde h|^{2}dx + \displaystyle\int_{0}^{1}\rho_{2}^{2}|G|^{2} dx\right)  \\
        \quad + \epsilon \displaystyle\int_{0}^{1}\hat{\rho}^{2} a(x) |y_{x}|^{2} dx + C_\epsilon \left( \int_0^1 |y|^2 dx \right).
    \end{array}
\end{equation*}

Since $b(t)$ is bounded and $\rho_0$ is bounded by below, there is a constant $D>0$ such that
\begin{equation*}
    \begin{array}{l}
       \dfrac{1}{2}\dfrac{d}{dt}\displaystyle\int_{0}^{1}\hat{\rho}^{2}(|y|^{2})dx + \displaystyle\int_{0}^{1}\hat{\rho}^{2} a(x)(|y_{x}|^{2})dx\vspace{0.1cm}\\
        \leq D\left(\displaystyle\int_{0}^{1}\rho_{0}^{2}(|y|^{2} )dx + \displaystyle\int_{\omega_1}\rho_{1}^{2}|\tilde h|^{2}dx + \displaystyle\int_{0}^{1}\rho_{2}^{2}|G|^{2}dx\right)
    \end{array}
\end{equation*}
and, integrating in time, we conclude \eqref{des Proposition 5}.

%--------------------------------------

Now, to prove \eqref{des Proposition 6}, we multiply $\eqref{eq:linearized_system new}$ by $\rho_{1}^{2} y_{t}$  and integrate over $[0,1]$. We get
\begin{equation*}%\label{third estimate}
    \begin{array}{l}
        \displaystyle\int_{0}^{1}\rho_{1}^{2}(|y_{t}|^{2})dx + \dfrac{1}{2} b(t)
        \displaystyle\int_{0}^{1}\rho_{1}^{2} a(x) \dfrac{d}{dt}(|y_{x}|^{2})dx % \vspace{0.1cm}\\
        + \displaystyle\int_0^1 \hat{\rho} c_1(x,t) y y_t dx 
        + \int_0^1 \hat{\rho} d_1(x,t) \sqrt{a}y_x y_t dx \\
        \quad \leq C\left(\displaystyle\int_{0}^{1}\rho_{0}^{2}(|y|^{2} )dx + \displaystyle\int_{\omega_1}\rho_{1}^{2}|\tilde h|^{2}dx + \displaystyle\int_{0}^{1}\rho_{2}^{2}|G|^{2} dx\right) %\vspace{0.1cm}\\
        + \ \dfrac{1}{4}\displaystyle\int_{0}^{1}\rho_{1}^{2} |y_{t}|^{2} dx.
    \end{array}
\end{equation*}

Thus, using that $\hat\rho = \rho_0 \rho_1$, Young's inequality and the boundedness of $c_i$ and $d_i$, we get
\begin{equation}\label{third estimate}
    \begin{array}{l}
        \displaystyle\int_{0}^{1}\rho_{1}^{2}(|y_{t}|^{2})dx + \dfrac{1}{2} b(t) \dfrac{d}{dt}\displaystyle\int_{0}^{1}\rho_{1}^{2} a(x)(|y_{x}|^{2})dx  \vspace{0.1cm}\\
         \leq D\left(\displaystyle\int_{0}^{1}\rho_{0}^{2}(|y|^{2} )dx + \displaystyle\int_{\omega_1}\rho_{1}^{2}|\tilde h|^{2}dx + \displaystyle\int_{0}^{1}\rho_{2}^{2} |G|^{2} dx\right) %\vspace{0.1cm}\\
        + \ \dfrac{1}{2}\displaystyle\int_{0}^{1}\rho_{1}^{2} |y_{t}|^{2} dx + |\widetilde{\mathcal{M}}|,
    \end{array}
\end{equation}
where $\widetilde{\mathcal{M}}= \dfrac{1}{2}\displaystyle\int_{0}^{1}(\rho_{1}^{2})_{t}\, a(x) |y_{x}|^{2} dx$.
Since $|\zeta_{t}|\leq C\zeta^{2}$,  
we have that $|(\rho^{2}_{1})_{t}|\leq C\hat{\rho}^{2}$. 
Hence,
$$
    |\widetilde{\mathcal{M}}| \leq C\displaystyle\int_{0}^{1}\hat{\rho}^{2}\, a(x) |y_{x}|^{2} dx.
$$ 
 
Thus, \eqref{third estimate} gives 
\begin{equation*}
    \begin{array}{l}
        \dfrac{1}{2}\displaystyle\int_{0}^{1}\rho_{1}^{2}(|y_{t}|^{2})dx + \dfrac{1}{2} b(t)  \dfrac{d}{dt}\displaystyle\int_{0}^{1}\rho_{1}^{2} a(x) |y_{x}|^{2} dx \\ % \vspace{0.1cm}\\
        \leq D \left(\displaystyle\int_{0}^{1}\rho_{0}^{2}(|y|^{2} )dx + \displaystyle\int_{\omega_1}\rho_{1}^{2}|\tilde h|^{2}dx + \displaystyle\int_{0}^{1}\rho_{2}^{2} |G|^{2} dx\right) %\vspace{0.1cm}\\
        +  C\displaystyle\int_{0}^{1}\hat{\rho}^{2}\, a(x) |y_{x}|^{2} dx.
    \end{array}
\end{equation*}

Integrating the previous inequality from $0$ to $t$ and using the boundedness of $b(t)$ and the first estimate, \eqref{des Proposition 5}, we get
\begin{equation}\label{primeira da des proposição 6}
    \begin{array}{c}
        \displaystyle\sup_{t \in [0,T]}(\rho_{1}^{2}\|\sqrt{a}y_{x} \|^{2}_{L^{2}(0,1)}) + \displaystyle\int_{Q}\rho_{1}^{2} |y_{t}|^{2} dxdt
        \leq C \kappa_{1}(G,y_{0}).
     \end{array}
\end{equation}

%--------------------------------------

Finally, to conclude \eqref{des Proposition 6}, it remains to estimate $\int_{Q}\rho_{1}^{2} |(a(x)y_{x})_{x}|^{2} dxdt$. To do this, it is enough to multiply $\eqref{eq:linearized_system new}$ by $-\rho_{1}^{2}(a(x)y_{x})_{x}$ and integrate over $[0,1]$ as before. We get
\begin{equation*}
    \begin{split}
        \frac{1}{2}\int_0^1 \rho_1^2 a(x) \frac{d}{dt} (|y_x|^2 ) dx + b(t) \int_0^1 \rho_1^2 \left( |(a y_x)_x|^2  \right) 
        - \int_0^1 \rho_1^2 \left( (a y_x)_x c_1 y  \right) - \int_0^1 \rho_1^2 \left( (a y_x)_x d_1 \sqrt{a} y_x  \right) \\
        \leq C \left( \int_0^1 \rho_1^2 \left( (|\tilde h| + |H|)|(a y_x)_x| \right) \right).
    \end{split}    
\end{equation*}

Thus, for a small $\epsilon>0$, 
\begin{equation*}
    \begin{array}{l}
         \dfrac{1}{2}\dfrac{d}{dt}\displaystyle\int_{0}^{1}\rho_{1}^{2} \,a(x) |y_{x}|^{2} dx + b(t) \displaystyle\int_{0}^{1}\rho_{1}^{2} |(a(x)y_{x})_{x}|^{2} dx %\vspace{0.1cm}
         \\
         \leq C_\epsilon\left(\displaystyle\int_{0}^{1}\rho_{1}^{2} |y|^{2} dx + \displaystyle\int_{\omega_1}\rho_{1}^{2}|\tilde h|^{2}dx + \displaystyle\int_{0}^{1}\rho_{1}^{2} |G|^{2} dx\right) %\vspace{0.1cm}
         \\
         \qquad+ \displaystyle \epsilon \int_{0}^{1}\rho_{1}^{2} |(a y_x)_x|^{2} dx
         %+ \displaystyle \epsilon \int_{0}^{1}\rho_{1}^{2}\left(|(a y_x)_x|^{2} \right)dx 
         + C_\epsilon \int_{0}^{1}\rho_{1}^{2} |c_1 y|^2 dx + |\mathcal{N}|,
    \end{array}
\end{equation*}
where $|\mathcal{N}| = \dfrac{1}{2}\displaystyle\int_{0}^{1}(\rho_{1}^{2})_{t}\, a(x) |y_{x}|^{2} dx$.
Since $|\zeta_{t}|\leq C\zeta^{2}$,  
we have that $|(\rho^{2}_{1})_{t}|\leq C\hat{\rho}^{2}$. 
Hence,
$$
    |\mathcal{N}| \leq C\displaystyle\int_{0}^{1}\hat{\rho}^{2}\, a(x) |y_{x}|^{2} dx.
$$ 

Thus, since $b(t)$,  $c_i$ and $d_i$ are bounded, and $\rho_1 \leq C\rho_0 \leq C\rho_2$, we get, for some $D>0$,
\begin{equation*}
    \begin{array}{l}
         \dfrac{1}{2}\dfrac{d}{dt}\displaystyle\int_{0}^{1}\rho_{1}^{2} \,a(x)(|y_{x}|^{2})dx + \dfrac{1}{2}\displaystyle\int_{0}^{1}\rho_{1}^{2}(|(a(x)y_{x})_{x}|^{2} )dx\vspace{0.1cm}\\
         \leq D \left(\displaystyle\int_{0}^{1}\rho_{0}^{2} |y|^{2} dx + \displaystyle\int_{\omega_1}\rho_{1}^{2}|\tilde h|^{2}dx + \displaystyle\int_{0}^{1}\rho_{2}^{2} |G|^{2} dx\right) %\vspace{0.1cm}\\
         +\, D\displaystyle\int_{0}^{1}\hat{\rho}^{2}\, a(x) |y_{x}|^{2} dx. 
    \end{array}
\end{equation*}

Integrating over time and using estimates (\ref{des Proposition 5}) and (\ref{primeira da des proposição 6})  we obtain
\begin{equation}\label{segunda da des proposição 6}
    \begin{array}{l}
        \displaystyle\int_{Q}\rho_{1}^{2} |(a(x)y_{x})_{x}|^{2} dxdt  \leq C \kappa_{1}(G,y_{0}).
    \end{array}
\end{equation}

From (\ref{primeira da des proposição 6}) and (\ref{segunda da des proposição 6}) we infer \eqref{des Proposition 6}.
\end{proof}

\section{Local null controllability of the nonlinear system}
\label{sec:control for nonlinear system}

%All along this section we use the weights defined in \eqref{eq:weights_rhos} and 
We use Liusternik's inverse function theorem  to obtain our local controllability results for the nonlinear system. Here $B_{r}(\zeta)$ denote an open ball of radius $r$ and centered at $\zeta$.

\begin{teo}[Liusternik \cite{Alekseev}]\label{Liusternik} 
    Let $ \mathcal{Y}$ and $ \mathcal{Z}$ be Banach spaces and let $\mathcal{A}:B_{r}(0)\subset  \mathcal{Y}\rightarrow  \mathcal{Z}$ be a $\mathcal{C}^{1}$ mapping. Let as assume that $\mathcal{A}^{\prime}(0)$ is onto and let us set $\mathcal{A}(0)=\zeta_{0}$. Then, there exist $\delta >0$, a mapping $W: B_{\delta}(\zeta_{0})\subset  \mathcal{Z}\rightarrow  \mathcal{Y}$ and a constant $K>0$ such that
    \begin{equation*}
        W(z)\in B_{r}(0),\,\, \mathcal{A}(W(z))=z\,\, \text{and}\,\, \Vert W(z)\Vert_{ \mathcal{Y}}\leq K\Vert z-\zeta_{0}\Vert_{ \mathcal{Z}}\, \, \forall\, z\in B_{\delta}(\zeta_{0}).
    \end{equation*}
    In particular, $W$ is a local inverse-to-the-right of $\mathcal{A}$.
\end{teo}

We define a map $\mathcal{A} : \mathcal{Y} \to \mathcal{Z}$ between suitable Banach spaces $\mathcal{Y}$ and $\mathcal{Z}$ whose definition involve the estimates of the controllability of the linearized system in Proposition \ref{prop:linear_control}.

From the linearized equation \eqref{eq:linearized_system1}, we denote
\begin{equation*}
    \begin{split}
        G=&z_t-b(t)\left({a}(x)z_x\right)_x-B(x,t)\sqrt{a}z_x + D_3F\left(\ell(t)x,t,\widetilde{y}\right)z - h \widetilde y \cara_{_{{\omega}_1}}.     
    \end{split}
\end{equation*}

Let us define the space
\begin{equation}\label{eq:espaceY}
    \begin{array}{c}
        \mathcal{Y} = \{  (z, h)\in [L^{2}(\Omega\times (0,T))]\times L^{2}( \omega_1 \times (0,T)) \ : \ z(\cdot,t) \ \text{is abs. continuous in}\ \Omega,\ \text{a.e. in}\ [0, T], \\
        \ \rho_1 h \widetilde y \in L^{2}( \omega_1 \times (0,T)), \
        \rho_{0} z, \ \rho_2 G \in L^2(Q), \\
        z(1, t) \equiv z(0,t) \equiv 0 \ \text{a.e in}\ [0, T], \ z(\cdot,0) \in H_a^1(\Omega) \}.
    \end{array}
\end{equation}

Thus, $\mathcal{Y}$ is a Hilbert space with the norm
\begin{equation*}
    \begin{array}{lll}
          \Vert (z, h)\Vert^{2}_{\mathcal{Y}} &=& \Vert \rho_{0} z\Vert^{2}_{L^{2}(Q)} +  \Vert\rho_1 h \widetilde y\Vert^{2}_{L^{2}(\omega_1 \times (0,T))} + \Vert\rho_{2} G\Vert^{2}_{L^{2}(Q)} + \Vert z(\cdot,0)\Vert^{2}_{H^{1}_{a}(\Omega)}.
    \end{array}
\end{equation*}

Due to Proposition  \ref{addicional_estimates_case_linear}, for any $(z, h)\in \mathcal{Y}$ we have:
\begin{equation}\label{eq:estimativas_total}
    \begin{array}{c}
        \displaystyle\sup_{[0,T]}(\hat{\rho}^{2}\|z\|^{2}_{L^{2}(0,1)}) 
        +\displaystyle\int_{Q}\hat{\rho}^{2} a(x) |z_{x}|^{2} dxdt
        + \displaystyle\sup_{[0,T]}(\rho_{1}^{2}\|\sqrt{a}z_{x} \|^{2}_{L^{2}(0,1)}) \\
        + \displaystyle\int_{Q}\rho_{1}^{2} (|z_{t}|^{2} + |(a(x)z_{x})_{x}|^{2}) dxdt
        \leq C \Vert(z, h)\Vert^{2}_{\mathcal{Y}}. 
    \end{array}
\end{equation}

Now, let us introduce the Banach space $\mathcal{Z} = \mathcal{F} \times H_a^1(\Omega)$ such that 
$$\mathcal{F}=\{ z \in L^2(Q) \ : \ \rho_2 z \in L^2(Q) \}.$$

Finally, consider the mapping $\mathcal{A} : \mathcal{Y} \to \mathcal{Z}$ such that
$(z, h) \mapsto (\mathcal{A}_1, \mathcal{A}_2)$ where the components $\mathcal{A}_i$, $i=1,2$, are given by
\begin{equation}\label{aplicação A}
    \left\{\begin{array}{ll}
        \mathcal{A}_1(z, h) =& z_t - b(t) \left( a(x) z_x \right)_x - B(t) c(x)\sqrt{a}z_x + F\left(\ell(t)x,t,z+\widetilde{y}\right) -F\left(\ell(t)x,t,\widetilde{y}\right) - {h}\cara_{_{{\omega_1}}}(z+\widetilde{y}) ,\\ %\vspace{0.1cm}\\
        \mathcal{A}_2(z, h) =& z(\cdot,0).
    \end{array}\right.
\end{equation}

We prove that we can apply Theorem \ref{Liusternik} to the mapping $\mathcal{A}$ through the following three lemmas:
\begin{lema}\label{A bem definido}
    Let $\mathcal{A}: \mathcal{Y}\rightarrow  \mathcal{Z}$ be given by \eqref{aplicação A}. Then, $\mathcal{A}$ is well defined and continuous. 
\end{lema}

\begin{proof}
	We want to show that $\mathcal{A}(z, h)$ belongs to $ \mathcal{Z}$, for every $(z, h) \in \mathcal{Y}$.
    
    Clearly $\| \mathcal{A}_{2}(z, h)\|_{\mathcal{F}}^2 < \infty$. Now, we show that $\mathcal{A}_{1}(z, h)$, belongs to its respective space. Indeed, 
	\begin{equation*}
		\begin{array}{lll}
			\|\mathcal{A}_{1}(z, h)\|^{2}_{\mathcal{F}}&\leq & 2\displaystyle\int_{Q}\rho^{2}_{2}|G|^{2}dxdt \\
            &&+ 2\displaystyle\int_{Q}\rho^{2}_{2}\left|F\left(\ell(t)x,t,z+\widetilde{y}\right)-F\left(\ell(t)x,t,\widetilde{y}\right) -D_3F\left(\ell(t)x,t,\widetilde{y}\right)z \right|^{2}dxdt \\
            && + 2\displaystyle\int_{Q}\rho^{2}_{2}\left| h\cara_{_{{\omega_1}}}z \right|^{2}dxdt \\
			&=:& 2 I_{1} + 2 I_{2} + 2 I_3.
		\end{array}
	\end{equation*}

    It is immediate by definition of the space $\mathcal{Y}$ that $I_{1}\leq C \|(z, h)\|^{2}_{\mathcal{Y}}$. Furthermore, 
    using the mean value theorem, for some $\tilde \theta = \theta(x,t) \in (0,1)$ and the properties of the weights \eqref{eq:compara_rhos}, we have 
    %\color{red}
    \begin{equation*}
        \begin{split}
            I_2 &\leq \int_Q \rho_2^2 \left|D_3 F(\ell(t)x,t, \widetilde{y} + \tilde \theta z) - D_3 F(\ell(t)x,t,\widetilde{y}) \right|^2 |z|^2 dxdt \\
            &\leq \int_Q \rho_2^2 \tilde\theta^2 |z|^2 |z|^2 dxdt \\
            & \leq \int_Q \rho_1^4 |z|^2 |z|^2 dxdt \\
            & \leq \int_0^T \sup_{x \in \Omega} \{ \rho_1^2 |z|^2 \} \int_\Omega \rho_0^2 |z|^2 dx dt.
        \end{split}    
    \end{equation*}
    
    Using the continuous immersion $H^1_a(\Omega) \subset L^\infty(\Omega)$, the weight comparison properties \eqref{eq:compara_rhos}, and \eqref{des Proposition 6}, we get 
    \begin{equation*}
        \begin{split}
            I_2 &\leq \int_0^T \| \rho_ 1^2 \sqrt{a} z_x(t)\|^2  \int_\Omega \rho_0^2 |z|^2 dx dt \\
            &\leq \sup_{t \in [0,T]} \left( \rho_1^2  \|\sqrt{a} z_x(t)\|^2  \right) \int_Q  \rho_0^2 |z|^2 dx dt \\
            &\leq \Vert(z, h)\Vert^{2}_{\mathcal{Y}} \Vert(z, h)\Vert^{2}_{\mathcal{Y}}.
        \end{split}    
    \end{equation*}

    For $I_3$, using the continuous immersion $H^1_a(\Omega) \subset L^\infty(\Omega)$, the weight comparison properties \eqref{eq:compara_rhos}, we estimate
    \begin{equation*}
        \begin{split}
            I_3 &\leq \int_0^T \rho_2^2 \|h\|_{L^4(\Omega)}^2 \|z\|_{L^4(\Omega)}^2 dxdt \\
            & \leq C\int_0^T \rho_1^4 \overline \rho^{-2} \rho_1^{-2} (\overline \rho^2 \| h_x \|^2) (\rho_1^2 \|\sqrt{a} z_x\|^2) dxdt \\
            & \leq C \int_0^T \|\overline \rho h_x\|^2 \sup_{t \in [0,T]} \|\rho_1 \sqrt{a} z_x\|^2 dxdt  \\
            &\leq C \Vert \overline \rho h\Vert^{2}_{\mathsf{U}} \Vert(z, h)\Vert^{2}_{\mathcal{Y}}.
        \end{split}
    \end{equation*}

	Thus, from the regularity result \eqref{cr2}, $I_3$ is bounded. Therefore, we conclude that $\mathcal{A}_{1}(z, h)\in \mathcal{F}$ and $\mathcal{A}$ is well-defined.
\end{proof}

\begin{lema}\label{DA continuo}
    The mapping $\mathcal{A}: \mathcal{Y}\longrightarrow  \mathcal{Z}$ is continuously differentiable.
\end{lema}

\begin{proof}
	First, we prove that $\mathcal{A}$ is Gateaux differentiable at any $(z, h) \in \mathcal{Y}$. 
    %Let us compute the $\textit{G-derivative}$ ${\mathcal{A}}^{\prime}(z, \tilde h)$.
	Consider the linear mapping $D \mathcal{A}: \mathcal{Y} \to \mathcal{Z}$ given by
	$D\mathcal{A}(z, h) = (D\mathcal{A}_1,D\mathcal{A}_2)$,	
	where for $(\bar z, \bar{h}) \in \mathcal{Y}$,
	\begin{equation}\label{eq:der_map_A}
		\begin{cases}
			D\mathcal{A}_1(\bar z, \bar{h}) = &  \, \bar z_t - b(t)(a(x) \bar z_x)_x -B\sqrt{a} \bar z_x + D_3 F\left(\ell(t)x,t,z+\widetilde{y}\right)\bar z - \bar{h} z \cara_{\omega_1} - h\bar z \cara_{\omega_1} - \bar h \widetilde y \cara_{\omega_1}, \\
    		D\mathcal{A}_2(\bar z, \bar{h}) =& \, \bar z(0).
		\end{cases}%\right.    
	\end{equation}
	
	We have to show that, for $i=1,2$, 
	$$
    	\frac{1}{\lambda}\left[ \mathcal{A}_i ((z, h)+\lambda(\bar z, \bar{h})) - \mathcal{A}_i (z, h) \right] \to D\mathcal{A}_i(\bar z, \bar{h}),
	$$
	strongly as $\lambda \to 0$, in the corresponding factor of $\mathcal{Z}$ .
    Indeed,
	\begin{equation*}
		\begin{split}
			&\left\| \frac{1}{\lambda}\left[ \mathcal{A}_1 ((z, h)+\lambda(\bar z, \bar{h})) - \mathcal{A}_1(z, h) \right] - D\mathcal{A}_1(\bar z, \bar{h}) \right\|^{2}_{L^2(\rho_2^2,Q)} \\
			&=\left\|
    			 \frac{1}{\lambda} \left[ z_t+\lambda \bar z_t - b(t)\left( a(x) (z_x+\lambda \bar z_x) \right)_x - B\sqrt{a}(z_x+\lambda \bar z_x) \right.\right. \\
            & \ \ \ \ \ \ \ \ \ \ \left.\left. + F\left(\ell(t)x,t,z+\lambda \bar z +\widetilde{y}\right) -F\left(\ell(t)x,t,\widetilde{y}\right) - (h+\lambda \bar{h})(z+\lambda \bar z + \widetilde y) \cara_{\omega_1} \right.\right. \\
			& \ \ \ \ \ \ \ \ \ \ \left.\left. - z_t + b(t)\left( a(x) z_x \right)_x + B\sqrt{a}z_x - F\left(\ell(t)x,t,z+\widetilde{y}\right) + F\left(\ell(t)x,t,\widetilde{y}\right) + h (z+\tilde y) \cara_{\omega_1} \right] \right.\\
			& \ \ \ \ \ \ \ \ \ \ \left. -\bar z_t + b(t)(a(x) \bar z_x)_x + B\sqrt{a} \bar z_x - D_3 F\left(\ell(t)x,t,z+\widetilde{y}\right)\bar z + \bar{h} z \cara_{\omega_1} + h\bar z \cara_{\omega_1} + \bar h \widetilde y \cara_{\omega_1} \right\|^{2}_{L^2(\rho_2^2,Q)} \\
			& = \int_Q \rho_2^2 \left| \frac{1}{\lambda}\left( F\left(\ell(t)x,t,z+\lambda \bar z +\widetilde{y}\right) -F\left(\ell(t)x,t,z+\widetilde{y}\right) \right) - D_3 F\left(\ell(t)x,t,z+\widetilde{y}\right)\bar z \right|^2 \\ %= J_1. 
			& \quad + \int_Q \rho_2^2 \left| \lambda \bar h \bar z \cara_{\omega_1} \right|^2  = J_1 + J_2.
		\end{split}
	\end{equation*}
    
    By the mean value theorem and using that $F$ is of class $C^2$,  
    for $\tilde\lambda = \tilde\lambda(x,t) \in (0,\lambda)$, we have 
	\begin{eqnarray*}
		J_1 &=&\int_Q \rho_2^2 \left| D_3 F\left(\ell(t)x,t,z + \tilde\lambda \bar z + \widetilde{y}\right) - D_3 F\left(\ell(t)x,t,z+\widetilde{y}\right) \right|^2 |\bar z|^2 \\
		&=& C_1 \textcolor{blue}{\lambda} \int_Q \rho_2^2 |\bar z|^4 , 
	\end{eqnarray*}

    By the same estimate as for $I_2$ in Lemma \ref{A bem definido}, we get that $J_1$ converges to zero as $\lambda \to 0$ using \eqref{eq:estimativas_total}. On the other hand, clearly $J_2 \to 0$ as $\lambda \to 0$.
    
    This finishes the proof that $\mathcal{A}$ is Gateaux differentiable, with a \textit{G-derivative} $D\mathcal{A}(z, h)$.
	
    Now, take $(z, h)\in \mathcal{Y}$ and let $((z_{n}, h_{n}))_{n=0}^{\infty}$ be a sequence that converges to  $(z, h)$ in $\mathcal{Y}$. 
    From the expression of the formal derivative of $\mathcal{A}$,  \eqref{eq:der_map_A}, we have 
	\begin{eqnarray*}
		&(D\mathcal{A}_1(z_{n}, h_{n}) - D\mathcal{A}_1(z, h))(\bar z, \bar{h})
        &= \left(D_3 F\left(\ell(t)x,t,z_n+\widetilde{y}\right) - D_3 F\left(\ell(t)x,t,z+\widetilde{y}\right) \right)\bar z  \\
        && \quad - \bar{h} (z_n-z) \cara_{\omega_1} + (h_n - h)\bar z \cara_{\omega_1} \\
        &&= X_1^1+X_2^1 + X_3^1.
	\end{eqnarray*}
    
	For $X^{1}_{1}$ we have, 
	\begin{eqnarray*}
		\int_{0}^{T}\int_{0}^{1}\rho^{2}_{2}|X_{1}^{1}|^{2} dxdt &\leq& C \int_{0}^{T}\int_{0}^{1}\rho_{2}^{2} \left| D_3 F\left(\ell(t)x,t,z_n+\widetilde{y}\right) - D_3 F\left(\ell(t)x,t,z+\widetilde{y}\right) \right|^2 |\bar z|^{2}dxdt\\
		&\leq&C\int_{0}^{T}\int_{0}^{1}\rho_{2}^{2} M^2 |z_n-z|^2 |\bar z|^{2}dxdt
	\end{eqnarray*}

    By \eqref{eq:compara_rhos} we have $\rho_2^2\leq C \rho_1^4 \leq C \rho_1^2 \rho_0^2$. As above, using the continuous immersion $H^1_a(\Omega) \subset L^\infty(\Omega)$, we get  
	\begin{eqnarray*}
		\int_{0}^{T}\int_{0}^{1}\rho^{2}_{2}|X_{1}^{1}|^{2} dxdt&\leq& C\int_{0}^{T}\int_{0}^{1} \rho_1^2 \rho_0^2  |z_n-z|^2 |\bar z|^{2}dxdt\\
        &\leq& \int_0^T \rho_1^2 \rho_0^2 \|z_n-z\|_{L^2(\Omega)} \|\bar z\|_{L^\infty(\Omega)} dt\\
		&\leq& C \sup_{t \in [0,T]}\left(\rho_1^2\|\sqrt{a} \bar{z}_x\|^2_{L^2(\Omega)}\right) \int_{0}^{T}\int_{0}^{1} \rho_0^2 |z_n-z|^2 dxdt\\
		&\leq&C\|(\bar z,\bar{h})\|_{\mathcal{Y}} \cdot \|\rho_0 (z_n-z)\|_{L^2(Q)}\\
		&\leq&C \|(\bar z,\bar{h})\|_{\mathcal{Y}} \|(z_{n}-z),(h_{n}- h)\|_{\mathcal{Y}}\rightarrow 0.
	\end{eqnarray*}
	
    On the other hand, for $X^{1}_{2}$, using the continuous immersion $H^1_a(\Omega) \subset L^\infty(\Omega)$, and the fact that $e^{sA^*}\zeta^k$ is bounded for any $k \in \mathbb{Z}$,
	\begin{eqnarray*}
		\int_{0}^{T}\int_{0}^{1}\rho^{2}_{2}|X_{2}^{1}|^{2} dxdt &\leq& C \int_{0}^{T}\int_{0}^{1}\rho_2^{2} |\bar{h}|^2 |z_n-z|^2  dxdt\\
		&\leq&C\int_{0}^{T}\rho_2^{2} \|z_n-z\|_{L^4(\Omega)}^2 \|\bar h\|_{L^4(\Omega)}^2 dxdt \\
        &\leq&C\int_{0}^{T} \rho_2^{2} \rho_1^{-2}\overline\rho^{-2} \rho_1^2 \|z_n-z\|_{L^4(\Omega)}^2 \overline\rho^2 \|\bar h\|_{L^4(\Omega)}^2 dt \\
        &\leq&C\int_{0}^{T} e^{sA^*}\hat\zeta^{-2}(\zeta^*)^8 (\rho_1^2 \|\sqrt{a} (z_n-z)_x \|^2) (\overline \rho^2 \|\sqrt{a} \bar h_x\|^2) dt \\
        &\leq&C \sup_{t \in [0,T]} (\rho_1^2 \|\sqrt{a} (z_n-z)_x \|^2) \int_{0}^{T}  \|\overline \rho \bar h_x \|^2 dt \\
        &\leq&C\|(z_n-z,h_n-h)\|_{\mathcal{Y}}^2 \|\overline \rho \bar h\|_{\mathsf{U}}^2.
	\end{eqnarray*}    

    The estimate for $X_3^1$ is analogous to the case $X_2^1$. Therefore, $(z, h) \longmapsto \mathcal{A}^{\prime}(z, h)$ is continuous from $\mathcal{Y}$ into $\mathcal{L}(\mathcal{Y},\mathcal{Z})$ and, consequently in view of classical results, we have that $\mathcal{A}$ is Fr\'echet-differentiable and $\mathcal{C}^{1}$.
    \color{black}
\end{proof}

\begin{lema}\label{Mapa sobrejetivo}
    Let $\mathcal{A}$ be the mapping in \eqref{aplicação A}. Then, $\mathcal{A}^{\prime}(0,0)$ is onto.
\end{lema}

\begin{proof}
    Let $(H, z_{0})\in \mathcal{Z}$. From Theorem \ref{prop:linear_control} we know there exists $(z, \tilde h)$ satisfying the linear equation (\ref{eq:linearized_system1}) and the estimates (\ref{estimate for solution}). Furthermore, we know that $z \in C^{0}([0,T];L^{2}(0,1))\cap L^{2}(0,T;H^{1}_{a}(0,1))$. Consequently, taking $h = \tilde h / \widetilde y$ in $\omega_1 \times (0,T)$, we have $(z,h) \in \mathcal{Y}$ and
    $$
        \mathcal{A}^{\prime}(0,0)(z,h)=(H,z_{0}).
    $$ 
\end{proof}

\noindent\textbf{Proof of Theorem 
\ref{th}.}
%\ref{thm:local_null_controllability}.}

    \noindent According to Lemmas \ref{A bem definido}-\ref{Mapa sobrejetivo} we can apply the Inverse Mapping Theorem (Theorem \ref{Liusternik}) and consequently there exists $\delta > 0$ and a mapping $W:B_{\delta}(0)\subset {\mathcal{Z}}\rightarrow {\mathcal{Y}}$ such that
    \begin{equation*}
        W(w)\in B_{r}(0)\,\,\, \text{and}\,\,\, {\mathcal{A}}(W(w))=w, \quad \forall w\in B_{\delta}(0).
    \end{equation*}
    Taking $(0,z_{0})\in B_{\delta}(0)$ and $(z,h)=W(0,z_{0})\in {\mathcal{Y}}$, we have
    \begin{equation*}
        {\mathcal{A}}(z,h)=(0,z_{0}).
    \end{equation*}
    Thus, we conclude that (\ref{nolineal}) is locally null controllable at time $T > 0$. In other words, $$z(\cdot,T)=0 \quad \text{in} \quad \Omega .$$

    Next, taking $y=z+\widetilde{y}$, then  (\ref{eqprin}) is satisfied and $y(\cdot,T)=\widetilde{y}(\cdot, T)$ in $\Omega$.

    Finally, using the diffeomorphism $(x,t) \rightarrow (\overline{x},t)$ from $Q$ to $\widehat{Q}$, one has $$u(\cdot,T)=\widetilde{u}(\cdot, T) \quad \text{in} \quad \Omega_T .$$ This implies that the
original equation (\ref{eq:PDE}) is controllable to trajectories and ends the proof of Theorem \ref{th}.
    
\qed

\section{Additional comments}
\label{sec:final_remarks}
In the present work, we have established an exact controllability result to \eqref{eq:PDE} with distributed controls, locally supported in space. This result can be generalized for other situations. Initially, using analogous techniques, one can consider nonlinear systems of the form
\begin{equation*}\label{op1}
\left\{
\begin{array}
    [c]{lll}%
    u_t - \left(a(\overline{x}) u_{\overline{x}}\right)_{\overline{x}} + F(\overline{x},t,u,u_x)=\widehat{h}\cara_{\widehat{\omega}}u & \mbox{in} &
    \widehat{Q}:=\displaystyle\bigcup_{0\leq t \leq T} \{\Omega_t \times \{t\} \},\\
    u(0,t)=u(\ell(t),t)=0 & \mbox{on} & \widehat{\Sigma}:= \displaystyle\bigcup_{0\leq t \leq T} \{\Gamma_t \times \{t\} \},\\
    u(\overline{x},0)=u_{0}(\overline{x}) & \mbox{in} & \Omega,
\end{array}
\right.  %
\end{equation*}

Other important topics arise from our current research:
\begin{itemize}
    \item 
    Exact controllability to the trajectories of degenerate equation with nonlocal nonlinearities:
    \begin{equation*}\label{op2}
    \left\{
    \begin{array}
        [c]{lll}%
        u_t - \left(\beta\left(\overline{x},\int_{0}^{\ell(t)}u \,\,d{\overline{x}} \right)u_{\overline{x}}\right)_{\overline{x}} + F(\overline{x},t,u)=\widehat{h}\cara_{\widehat{\omega}}u & \mbox{in} &
        \widehat{Q}:=\displaystyle\bigcup_{0\leq t \leq T} \{\Omega_t \times \{t\} \},\\
        u(0,t)=u(\ell(t),t)=0 & \mbox{on} & \widehat{\Sigma}:= \displaystyle\bigcup_{0\leq t \leq T} \{\Gamma_t \times \{t\} \},\\
        u(\overline{x},0)=u_{0}(\overline{x}) & \mbox{in} & \Omega,
    \end{array}
    \right.  %
    \end{equation*}
    where $\beta$ is a separated variables function given by $\beta(\overline{x},r)=\mu (r)a(\overline{x})$ such that $\mu:\mathbb{R} \rightarrow \mathbb{R}$ is a $C^1$ function with bounded derivative. The function $\beta$ defines an operator which degenerates at $\overline{x}=0$ and has a nonlocal term. More precisely, the function $a$ behaves ${\overline{x}}^{\alpha}$, with $\alpha \in (0,1)$.  

    \item 
    An interesting case deals with the controllability of one-phase Stefan-like problems with the following structure:
    \begin{equation*}\label{op3}
    \left\{
    \begin{array}
        [c]{lll}%
        u_t - \left(a(\overline{x}) u_{\overline{x}}\right)_{\overline{x}} + F(\overline{x},t,u)=\widehat{h}\cara_{\widehat{\omega}}u, &  &
        (\overline{x},t) \in Q_L ,\\
        u(0,t)=u(L(t),t)=0, &  & t \in (0,T),\\
        u(\overline{x},0)=u_{0}(\overline{x}), &  & (\overline{x},t) \in (0,L_0).
    \end{array}
    \right.  %
    \end{equation*}
    $$
    a(\overline{x})u_{\overline{x}}(L(t),t)=-L'(t),\quad t \in (0,T),
    $$
    where $Q_L$ stands for the following set:
    $$
    Q_L=\{ (\overline{x},t);\,\, \overline{x}\in (0,L(t)),\,\,t \in (0,T)\},
    $$
    with $0<L_0<L(t)<B,\,\,t \in (0,T)$.
\end{itemize}

\vspace{0.3cm}
Some of these extensions will be considered in the near future.

\vspace{0.8cm}

\noindent\textbf{Acknowledgments}

This study was financed in part by the Coordenação de Aperfeiçoamento de Pessoal de Nível Superior - Brasil (CAPES) - Finance Code 001. A.S.G.  and 
L.Y. were partially supported by CAPES-Brazil.

\bibliographystyle{abbrv}
\bibliography{referencias}

\end{document}